\documentclass[12pt,reqno]{amsart}

\usepackage{AKstyle}

\begin{document}

\author{Jean Bourgain}
\thanks{Bourgain is partially supported by NSF grant DMS-0808042.}
\email{bourgain@ias.edu}
\address{IAS, Princeton, NJ}
\author{Alex Kontorovich}
\thanks{Kontorovich is partially supported by  NSF grants DMS-0802998 and DMS-0635607, and the Ellentuck Fund at IAS}
\email{avk@ias.edu}
\address{IAS and Brown, Princeton, NJ}
\author{Peter Sarnak}
\thanks{Sarnak is partially supported by NSF grant DMS-0758299.}
\email{sarnak@math.princeton.edu}
\address{IAS and Princeton, Princeton, NJ}

\title{Sector Estimates for Hyperbolic Isometries}


\begin{abstract}
We prove various orbital counting statements for Fuchsian groups of the second kind. These are of independent interest, and also are used in the companion paper \cite{BourgainKontorovich2009} to produce primes in the Affine Linear Sieve.
\end{abstract}
\date{\today}
\maketitle
\tableofcontents

\section{Introduction}

Let $\G<G=\PSL(2,\R)$ be a non-elementary,
geometrically finite group of isometries of the upper half plane $\bH$.
In the case when $\G$ has 
fundamental domain with finite hyperbolic area, much effort has gone
into understanding the asymptotic behavior of the number of points in a $\G$-orbit which lie in an expanding region inside $G$, e.g. 
\cite{
Selberg1965, 
LaxPhillips1982, 
GoodBook, 
DukeRudnickSarnak1993, EskinMcMullen1993, 
Kontorovich2009, KontorovichOh2009}.  
Such regions can be decomposed into 
their harmonics, and 
hence one can recover many counting statements from, say, taking the Cartan decomposition $G=KA^{+}K$ and studying the set of $\g\in\G$ in a ball of expanding radius with a given harmonic on right and left $K$-types.
Here 
$$
K=\SO(2)
=
\left\{
k_{\gt}=
\mattwo{\cos\gt}{\sin\gt}{-\sin\gt}{\cos\gt}
:\gt\in[0,2\pi)
\right\}
$$
 and 
 $$
 A^{+}=
 \left\{
 a_{t}=
 \mattwo{e^{-t/2}}00{e^{t/2}}
 ,t\ge0
 \right\}
.
 $$ 
\

The state of the art in this direction in the finite-volume case is due to Good \cite{GoodBook}. Let $0\le \gt_{1}(g),\gt_{2}(g)<\pi$ and $t(g)>0$   be the functions in the Cartan decomposition of $G$, so that $g=\pm k_{\gt_{1}(g)}\,a_{t(g)}\,k_{\gt_{2}(g)}$.%
\footnote{
These are determined as follows. Assume that $g
\notin K$.
Define $t(g)$ as the solution to $\|g\|^{2}=e^{t}+e^{-t}=2\cosh t$ with $t>0$.
Let $\gt_{1}(g)$ be defined as $\foh\arg\left({g\cdot i - i \over g\cdot i +i}\right)$ (that is, map $g$  to the unit disk $\bD$;
the argument of the image determines $\gt_{1}$).  Then
$a_{t(g)}^{-1}k_{\gt_{1}(g)}^{-1}g$ is in $K$, whence
 $\gt_{2}$ is 
determined.
}
\begin{thm}[\cite{GoodBook}, Thm 4]
Let $\G<G$ be a lattice 
and 
fix integers $n$ and $k$.
Let 
$$
0=\gl_{0}<\gl_{1}\le\cdots\le \gl_{N}<1/4
$$ 
be the eigenvalues of the Laplacian on $\G\bk\bH$ below $1/4$. Write 
\be\label{sjDef}
\gl_{j}=s_{j}(1-s_{j})
\ee
with $s_{j}>1/2$. Then
there are constants $c_{0}
,\dots,c_{N}
\in\C$ depending on $n$ and $k$ such that
\be\label{Goods}
\sum_{\g\in\G\atop\|\g\|<T}
e^{2in\,\gt_{1}(\g)}
\ 
e^{2ik\,\gt_{2}(\g)}
=
\bo_{n=0}\,\bo_{k=0} \, c_{0}\, T^{2}
+
\sum_{j=1}^{N}
c_{j}(n,k)T^{2s_{j}}
+
O_{n,k}(T^{4/3})
,
\footnote{
The exponent $4/3$ has never been improved, even for any specific $\G$.
}
\ee
as $T\to\infty$. 
\end{thm}

Our main goal 
is to give a version of the above in the case when $\G$ is 
again non-elementary and geometrically finite, but
whose fundamental domain has infinite area. First we recall the spectral theory in this context.
Given any fixed point $z\in\bH$, the orbit $\G z$ accumulates only on the boundary $\dd\bH\cong\hat\R$.
The set $\gL$ of accumulation points is called the limit set of $\G$, and is a Cantor-like set having some Hausdorff dimension 
$$
0<\gd<1
,
$$ 
which is called the critical exponent of $\G$.
 The spectrum above $1/4$ is purely continuous and there are finitely many discrete eigenvalues below $1/4$ \cite{Patterson1975
 }. In fact, the spectrum contains no point eigenvalues at all unless $\gd>1/2$, in which case the base eigenvalue is \cite{Patterson1976}
$$
\gl_{0}=\gd(1-\gd)
.
$$

Corresponding to $\gl_{0}$ is the base eigenfunction $\varphi_{0}$, which can be realized explicitly as the integral of a Poisson kernel against the so-called Patterson-Sullivan measure $\mu$ \cite{Patterson1976,Sullivan1984}, supported on the limit set $\gL\subset\dd\bH$. 
Roughly speaking, $\mu$ is the weak$^{*}$ limit as $s\to\gd^{+}$ of the measures
\be\label{muPS}
\mu_{s}(x):=
{
\sum_{\g\in\G}\exp({-s\,d(\fo,\g\cdot\fo)){\bf 1}_{x=\g\fo}}
\over
\sum_{\g\in\G}\exp({-s\,d(\fo,\g\cdot\fo))}
}
.
\ee
Here $d(\cdot,\cdot)$ is the hyperbolic distance, and $\fo$ is the origin (or any base point) in $\bD$. 

Let $\hat\mu$ denote the 
 Fourier coefficients 
of the measure $\mu$. Our first 
result is

\begin{thm}\label{thm:main}
Let $\G$ be a Fuchsian group of the second kind with critical exponent $\gd>1/2$.
Let 
\be\label{glLab}
0<\gd(1-\gd)=\gl_{0}<\gl_{1}\le\cdots\le \gl_{N}<1/4
\ee 
be the eigenvalues of the Laplacian on $\G\bk\bH$ below $1/4$, and 
use the notation \eqref{sjDef}.
Then 
for integers $n$ and $k$,
there are constants $c_{1},\dots,c_{N}\in\C$ depending on $n$ and $k$ such that 
\beann
\sum_{\g\in\G\atop\|\g\|<T}
e^{2in\,\gt_{1}(\g)}
\ 
e^{2ik\,\gt_{2}(\g)}
&=&
\hat\mu(2n)
\,
\hat\mu(2k)\
\pi^{1/2}
{\G(\gd-1/2)
\over
\G(\gd+1)}\
 T^{2\gd}
+
\sum_{j=1}^{N}
c_{j}(n,k)T^{2s_{j}}
 \\
 &&
 \quad
+
O\bigg(T^{1\cdot\frac14+2\gd\cdot\frac34}(\log T)^{1/4} (1+|n|+|k|)^{3/4}\bigg)
,
\eeann
as $T\to\infty$. Here $|c_{j}(n,k)|\ll |c_{j}(0,0)|$, as $n$ and $k$ vary, and the implied constants depend only on $\G$.
\end{thm}

\begin{rmk}\label{rmkLoss1}
{\rm
We make no attempt to obtain a best-possible error term; the above
 can surely be improved with some effort.
 The natural remainder term  here would be the one which corresponds to Lax-Phillips \cite{LaxPhillips1982} when $n=k=0$, namely $T^{1\cdot\frac23+2\gd\cdot\frac13}$, ignoring logs.
}
\end{rmk}

\begin{rmk}
{\rm
It is crucial for our intended applications below and in \cite{BourgainKontorovich2009} (for reasons of positivity) that the leading order term be recognized in terms of the Patterson-Sullivan measure; this is why we made the constant $c_{0}(n,k)$ completely explicit. 
}
\end{rmk}

\begin{rmk}
{\rm
In the absence of an explicit spectral expansion into Maass forms and Eisenstein series, we control the above error term using representation theory, smoothing the counting function in two copies of $\G\bk G$ and appealing to the decay of matrix coefficients \cite{HoweMoore1979,Cowling1978}. This technique dates at least as far back as \cite{DukeRudnickSarnak1993}.
}
\end{rmk}

In the intended applications in the companion paper \cite{BourgainKontorovich2009}, one requires the above with uniform control on cosets of congruence subgroups. Recall the spectral gap property in the infinite volume situation. Assume $\G<\SL(2,\Z)$. 
There is a fixed integer  $\fB\ge1$ called the {\it ramification number} which depends only on $\G$ and needs to be avoided. Let $q\ge1$ with 
$$
q=q'q'',\qquad\qquad\text{ and }\qquad\qquad q'\mid \fB
.
$$
Let $\G(q)$ denote a ``congruence'' subgroup of $\G$ of level $q$, that is, a group which contains the set
$$
\{\g\in\G:\g\equiv I(q)\}.
$$
The inclusion of vector spaces
$$
L^{2}(\G\bk\bH)
\subset
L^{2}(\G(q)\bk\bH)
$$
induces the same inclusion on the spectrum:
$$
\Spec(\G\bk\bH)
\subset
\Spec(\G(q)\bk\bH)
.
$$

\begin{Def}

{\rm
The {\it new spectrum} 
$$
\Spec_{new}(\G(q)\bk\bH)
$$ 
at level $q$ is defined to be the set of eigenvalues 
below
$
1/4
$
which are in
$
\Spec(\G(q)\bk\bH)
$
but  not in
$
\Spec(\G\bk\bH)
.
$
}
\end{Def}

\begin{Def}\label{fBdef}
{\rm
We will call 
$
\gT
$ a {\it spectral gap} for $\G$
 if  $\gT$ is in the interval $1/2<\gT<\gd$ 
and
$$
\Spec(\G(q)\bk\boldH)_{new}
\cap
(0,\gT(1-\gT)) 
\quad
\subset 
\quad
\Spec(\G(q')\bk\boldH)_{new}.
$$
}
\end{Def}

That is, the eigenvalues below $\gT(1-\gT)$ which are new for $\G(q)$ are coming from the ``bad'' part $q'$ of $q$. As the ramification number $\fB$ is a fixed integer depending only on $\G$, there are only finitely many possibilities for its divisors $q'$.

Infinite volume spectral gaps are known 
\cite{Gamburd2002,BourgainGamburd2007,BourgainGamburdSarnak2009} 
for prime and square-free $q$. The method in \cite{Gamburd2002} applies also for arbitrary composite $q$, 
and in particular
we have:

\begin{theorem}[\cite{Gamburd2002}]\label{BGSspec}
Let $\G$ be a Fuchsian group of the second kind with $\gd>5/6$. Then there exists some ramification number $\fB$ depending on $\G$ such that 
$\gT=5/6$
is a spectral gap for $\G$.
\end{theorem}

We also require the following Sobolev-type norm. Fix $T\ge1$ and 
let $\{X_{1},X_{2},X_{3}\}$ be a basis for the Lie algebra $\fg$, cf. \eqref{hef}. 
Then
%
define the  $\cS_{\infty,T}$ norm by
$$
\cS_{\infty,T}f=\max_{X\in\{0,X_{1},X_{2},X_{3}\}}\sup_{g\in G, \|g\|<T}|d\pi(X).f(g)|
,
$$
that is, the supremal value of first order derivatives of $f$ in a ball of radius $T$ in $G$.

Equipped with a spectral gap, we prove the following uniform counting statement.

\begin{thm}\label{thm:unif}
Let $T\ge1$ and $f:G\to\C$ be a smooth function with $|f|\le1$. 
Let $\G$ be as above
with $\gd>1/2$, ramification number $\fB$, and spectral gap $\gT$. Then for any $\g_{0}\in\G$ and  integer $q\ge1$ with $q=q'q''$ and $q'|\fB$,
\beann
\sum_{\g\in \g_{0}\cdot\G(q)\atop\|\g\|<T} f(\g)
&=&
{1\over [\G:\G(q)]}
\left(
\sum_{\g\in \G\atop\|\g\|<T} f(\g)
+\cE_{q'}
\right)
\\
&&
+
O\bigg(
(1+\cS _{\infty,T} f)^{6/7}T^{\frac{6}7 \cdot 2\gd+\frac17\cdot 2\gT}
\bigg)
.
\eeann
Here $\cE_{q'}\ll T^{2\gd-\ga_{0}}$, with $\ga_{0}>0$, and all implied constants are independent of $q''$ and $\g_{0}$. 
\end{thm}

\

The application in the companion paper \cite{BourgainKontorovich2009} requires the following two estimates, which we derive as consequences of the above. 

\begin{thm}\label{thm:LowBnd}
Assume $\G$ has critical exponent 
$\gd>1/2$.
Let $v$ and $w$ be 
vectors in $\Z^{2}$, 
and $n\in\Z$. Let $0<K<T<N$ be  parameters with $K\to\infty$ and $N/K\to\infty$ (and {\it a fortiori} $N/T\to\infty$).
Assume that $\frac NK<|n|<N$, $|w|< {N\over T}$, $|v|\le 1$,
and
$
|n|<|v||w|T.
$
Then
$$
\sum_{\g\in\G\atop\|\g\|<T}
\bo{\left\{
|\<v\g,w
\>-n|<{N\over K}
\right\}}
\gg
{T^{2\gd}\over K}
+ O\left(T^{\frac 34 + 2\gd \frac 14}(\log T)^{1/4}\right)
.
$$
\end{thm}
\

\begin{thm}\label{thm:UpBnd}
Assume $\G$ has $\gd>1/2$ 
and  spectral gap $\gT$.
Let $N, K,$ and $T$ be as above, and fix $q\ge1$.
Fix  $y=(y_{1},y_{2})$ and $(c,d)$ in $\Z^{2}$ 
such that
 $|y|< N$, $|(c,d)|<\frac NT$ and 
 $|y|<T|(c,d)|$. 
Then
$$
\sum_{\g\in\G\atop \|\g\|<T}
\bo
\left\{|(c,d)\g-y|<\frac {N}K\right\}
\bo
\bigg\{(c,d)\g\equiv y(\mod q) \bigg\}
\ll
{
T^{2\gd}
\over
K^{1+\gd} q^{2}
}
+
T^{\frac67 2\gd + \frac17 2\gT}
,
$$
as $N,K,T\to\infty$.
\end{thm}

In \S\ref{sec2}, we collect various preliminary pieces of information before proving Theorem \ref{thm:main} 
in \S\ref{sec3}
 and Theorem  \ref{thm:unif} 
in \S\ref{sec4}.
Finally, Theorem \ref{thm:LowBnd} is proved in \S\ref{sec:LowBnd} and Theorem \ref{thm:UpBnd} is proved in \S\ref{sec:UpBnd}.

\subsection*{Acknowledgements}
The second-named author wishes to express his gratitude to Stephen D. Miller for many helpful conversations.

\section{Preliminaries}\label{sec2}

\subsection{Representation Theory}
\

Let $G=\SL(2,\R)$, let $\G$ be a Fuchsian group of the second kind with critical exponent $\gd>1/2$,  use the notation \eqref{glLab} and \eqref{sjDef}, and let $\varphi_{j}$ be a Laplace eigenfunction corresponding to $\gl_{j}=s_{j}(1-s_{j})$.
The 
decomposition 
into irreducibles 
of the
right regular representation 
on the vector space $V=L^{2}(\G\bk G)$ 
is of the form \cite{GelfandGraevPS1966}
\be\label{Vdecomp}
V=
V_{\varphi_{0}}
\oplus
V_{\varphi_{1}}
\oplus
\cdots
\oplus
V_{\varphi_{N}}
\oplus
V_{\rm temp}
,
\ee
where each $V_{\varphi_{j}}$ is the $G$-span of the eigenfunction $\varphi_{j}$, and   is isomorphic as a $G$-representation to the complementary series representation with parameter $s_{j}$ (in our normalization, the principal series representations lie on the critical line $\Re(s)=1/2$). The reducible space $V_{\rm temp}$ consists of the tempered 
spectrum.

It will be convenient to use both the automorphic model above and, say, the line model, which we recall now.
Fix 
$s>1/2$
and let
 $(\pi,V_{s})$  denote the line model for the complementary series representation with parameter $s>1/2$ \cite{GelfandGraevPS1966}. That is, let $G$ act on functions $f:\R\to\C$ with action given by 
$$
\pi\mattwo abcd.f(x) = |-bx+d|^{-2s}f\left({ax-c\over -bx+d}\right).
$$
The intertwining operator $\cI:V_{s}\to V_{1-s}$ is defined by
\be\label{eq:intertwin}
\cI.f(y):= \int_{\R}{f(x)\over |x-y|^{2(1-s)}}dx.
\ee
For $f_{1},f_{2}\in V_{s}$, the pairing is 
\be\label{eq:pairing}
\<f_{1},f_{2}\>=\int_{\R}f_{1}(x) \overline{\cI.f_{2}(x)}dx. 
\ee
Then $V_{s}$ consists of functions $f$ with $\<f,f\><\infty$.

\subsection{Raising, Lowering, and Casimir Operators}
\

We return to the automorphic model and
let $\cH$ be one such irreducible $V_{\varphi_{j}}$. The dense subspace $\cH^{\infty}$ of smooth vectors is infinite dimensional, and decomposes further into one-dimensional $K$-isotypic components:
\be\label{eq:grading}
\cH^{\infty} = \bigoplus_{k\in\Z} \cH^{(2k)}
,
\ee
each $\cH^{(2k)}=\C\cdot v_{2k}$ consisting of functions of ``weight $2k$'', that is, those functions $v_{2k}\in\cH^{\infty}$ which transform as:
$$
v_{2k}(gk_{\gt})
=
e^{2ik\gt}
v_{2k}(g).
$$

Let 
\be\label{hef}
h=\mattwo100{-1},
\qquad
e=\mattwo0100,
\quad
\text{and}
\quad
f=\mattwo0010
\ee
denote a 
basis for the Lie algebra $\fg=\frak{sl}(2,\R)$.  The ladder (raising and lowering) operators $\cR$ and $\cL$ in the complexified Lie algebra $\fg_{\C}$ are defined by:
$$
\cR=h+i(e+f),
\qquad\qquad
\cL=h-i(e+f).
$$
Recall also the Casimir element $\cC$, which generates the center of the universal enveloping algebra $\cU(\fg_{\C}),$ and acts on $\cH$ as scalar multiplication by $-2\gl=-2s(1-s)$: 
$$
\cC
=
\foh h^{2}+ef+fe
.
$$

We will require expressions for these operators in Cartan coordinates $(\gt_{1},t,\gt_{2})$, corresponding to $g=k_{\gt_{1}}\,a_{t}\,k_{\gt_{2}}$.
For the Casimir operator, these can be found in many places, 
 e.g. \cite[p. 216 and p. 700]{Knapp1986}, \cite[p. 884]{CasselmanMilicic1982}, or \cite[\S7]{KnappTrapa2000}:
\be\label{Cis}
\foh\
{\cC} 
= 
{\dd ^{2 }
\over \dd t^{2}}
+
(\coth t)
{\dd  
\over \dd t }
+
\csch^{2}t
\left(
{\dd  ^{2}
\over \dd \gt_{1}^{2} }
+
{\dd  ^{2}
\over \dd \gt_{2}^{2} }
\right)
-
{2\cosh t
\over
\sinh^{2}t
}
{\dd  ^{2}
\over \dd \gt_{1}\dd\gt_{2} }
.
\ee

\

The raising and lowering operators are not as readily available in Cartan coordinates in the literature, so we derive their expression here.
\begin{lem}
In $KA^{+}K$ coordinates, the raising and lowering operators are:
$$
\cR
=
e^{2 i \gt _2}
\bigg(
-i  \csch(t)
{\dd\over \dd \gt_{1}}
+
2
{\dd\over \dd t}
+
i
\coth (t)
{\dd\over \dd \gt_{2}}
\bigg)
,
$$
and 
$$
\cL
=
e^{-2 i \gt _2}
\bigg(
i  \csch(t)
{\dd\over \dd \gt_{1}}
+
2
{\dd\over \dd t}
-
i
\coth (t)
{\dd\over \dd \gt_{2}}
\bigg)
.
$$
\end{lem}

\pf

Set $V=e+f$ and $Y=e-f$, so that $e=\foh V+ \foh Y$ and $f=\foh V - \foh Y$.
Let $\cJ$ be the Cartan decomposition, that is, the injection 
$$
\cJ:K\times A^{+}\times K\to G
.
$$ 
For a point $x=(\gt_{1},t,\gt_{2})$, we compute the derivation 
$$
D({\cJ}):T_{x}\to T_{g}=\fg
,
$$ 
as follows. Any element $T_{x}$ is of the form
$$
T_{x}=\xi_{1}Y_{1}+\eta h + \xi_{2}Y_{2},
$$
where $Y_{1}=Y=Y_{2}\in\fk=\frak {so}(2)$, but each is interpreted as a vector field in its respective component. We compute keeping only first order terms:

$$
\hskip-2in\cJ(k_{1}(I+\xi_{1}Y_{1}),\ a_{t}(I+\eta h), \ k_{2}(I+\xi_{2}Y_{2}))
$$
\beann
&
=&
k_{1}(I+\xi_{1}Y_{1}) a_{t}(I+\eta h) k_{2}(I+\xi_{2}Y_{2})
\\
&
=&
k_{1}a_{t}k_{2}(k_{2}^{-1}a_{t}^{-1}(I+\xi_{1}Y_{1}) a_{t} k_{2})(k_{2}^{-1}(I+\eta h) k_{2})(I+\xi_{2}Y_{2})
\\
&
=&
k_{1}a_{t}k_{2}(I+\ad(k_{2}^{-1})\ad(a_{t}^{-1})\xi_{1}Y_{1}) (I+\ad(k_{2}^{-1})\eta h)(I+\xi_{2}Y_{2})
\\
&
\approx&
k_{1}a_{t}k_{2}(I+\ad(k_{2}^{-1})\ad(a_{t}^{-1})\xi_{1}Y_{1} +\ad(k_{2}^{-1})\eta h +\xi_{2}Y_{2})
\\
&
=&
k_{1}a_{t}k_{2}\Bigg(I+\ad(k_{2}^{-1})\bigg\{\ad(a_{t}^{-1})\xi_{1}Y_{1} +\eta h +\ad(k_{2})\xi_{2}Y_{2}\bigg\}\Bigg)
.
\eeann
Hence
$$
D({\cJ}):\xi_{1}Y_{1}+\eta h + \xi_{2}Y_{2}\quad\mapsto\quad
\ad(k_{2}^{-1})\bigg\{\ad(a_{t}^{-1})\xi_{1}Y_{1} +\eta h +\ad(k_{2})\xi_{2}Y_{2}\bigg\}
.
$$
We compute 
\beann
\ad(a_{t}^{-1})Y
&=&
a_{t}^{-1} Y a_{t}
=
e^{-t}e-e^{t}f
=
e^{-t}
(\foh V + \foh Y)
-e^{t}
(\foh V - \foh Y)
\\
&=&
 \cosh (t)Y- \sinh (t)V
,
\\
\ad(k^{-1})Y
&=&
Y,
\\
\ad(k_{\gt}^{-1})V
&=&
 \cos (2 \gt )V- \sin (2 \gt )h
,
\\
\ad(k_{\gt}^{-1})h
&=&
 \sin (2 \gt )V+ \cos (2 \gt )h
.
\eeann
Therefore

\beann
&&
\hskip-.5in
D({\cJ}):
\xi_{1}Y_{1}+\eta h + \xi_{2}Y_{2}
\\
&\mapsto&
\ad(k_{2}^{-1})\bigg\{\ad(a_{t}^{-1})\xi_{1}Y_{1} +\eta h +\ad(k_{2})\xi_{2}Y_{2}\bigg\}
\\
&=&
\ad(k_{2}^{-1})\bigg\{\xi_{1}( \cosh (t)Y- \sinh (t)V) +\eta h +\xi_{2}Y\bigg\} 
\\
&=&
\ad(k_{2}^{-1})\bigg\{-\xi_{1} \sinh (t)V +\eta h +(\xi_{2}+\xi_{1} \cosh (t))Y\bigg\}
\\
&=&
-\xi_{1} \sinh (t)
\bigg( \cos (2 \gt_{2} )V- \sin (2 \gt_{2} )h
\bigg)
\\
&&
 +\eta 
\bigg( \sin (2 \gt_{2} )V+ \cos (2 \gt_{2} )h
\bigg)
 +(\xi_{2}+\xi_{1} \cosh (t))
 Y
\\
&=&
\bigg(
-\xi_{1} \sinh (t)\cos (2 \gt_{2} )
+\eta\sin (2 \gt_{2} )
\bigg)
V
\\
&&
+
\bigg(
\xi_{1} \sinh (t)
  \sin (2 \gt_{2} )
 +\eta 
\cos (2 \gt_{2} )
\bigg)
h
 +
 \bigg(
 \xi_{2}+\xi_{1} \cosh (t)
 \bigg)
 Y
 .
\eeann
To determine which $(\xi_{1},\eta,\xi_{2})$ give $\cR=h+iV$, we simply solve the linear  system  of equations:
\beann
-\xi_{1} \sinh (t)\cos (2 \gt_{2} )
+\eta\sin (2 \gt_{2} )
&=&i
\\
\xi_{1} \sinh (t)
  \sin (2 \gt_{2} )
 +\eta 
\cos (2 \gt_{2} )
&=&1
\\
\xi_{2}+\xi_{1} \cosh (t)
&=&
0
,
 \eeann
which has the solution:
\beann
\xi_{1}
&=&
-i e^{2 i \gt _2} \text{csch}(t)
   \\
   \eta
&=&
e^{2 i \gt _2}
\\
\xi_{2}
&=& 
ie^{2i\gt_{2}}\coth (t)
   .
\eeann
Of course $Y_{1}={\dd\over \dd \gt_{1}}$, $h=2{\dd\over \dd t}$, and $Y_{2}={\dd\over \dd\gt_{2}}$,
whence
$$
D({\cJ}):
\bigg(
-i e^{2 i \gt _2} \csch(t)
{\dd\over \dd \gt_{1}}
+
2
e^{2 i \gt _2}
{\dd\over \dd t}
+
ie^{2i\gt_{2}}\coth (t)
{\dd\over \dd \gt_{2}}
\bigg)
\mapsto 
\cR
.
$$
The formula for $\cL$ is derived in the same way.
\epf
\

\subsection{Polar Coordinates in the Disk Model}
\

At times it will also be convenient to use polar coordinates 
$(\gt_{1},r,\gt_{2})$ in  the unit tangent bundle of the disk model $\bD$, obtained from Cartan coordinates $(\gt_{1},t,\gt_{2})$ by 
the change of variables 
\be\label{rTot}
r={e^{t}-1\over e^{t}+1}=\tanh (t/2),\qquad {1+r\over 1-r}=e^{t},
\ee
with
$$
{\dd\over \dd t}={\dd r\over \dd t}{\dd\over \dd r}
={2e^{t}\over (e^{t}+1)^{2}}{\dd\over \dd r}
=\foh\sech^{2}(t/2){\dd\over \dd r}
=\foh(1-r^{2}){\dd\over \dd r}
,
$$
and
\beann
\csch(t)
=
\frac{1-r^2}{2 r}
,
\qquad
\coth(t)
=
\frac{1+r^2}{2 r}
.
\eeann

In the $(\gt_{1},r,\gt_{2})$, coordinates,  the Casimir operator becomes:
\bea\label{CisR}
\foh
\ {\cC} 
&=& 
{
(1-r^{2})^{2}
\over 
4
}
{\dd^{2}
\over
\dd r^{2}}
+
{
(1-r^{2})^{2}
\over 
4r
}
{\dd
\over
\dd r}
\\
\nonumber
&&
+
{
(1-r^{2})^{2}
\over
16 r^{2}
}
\bigg(
{\dd^{2}
\over
\dd \gt_{1}^{2}}
+
{\dd^{2}
\over
\dd \gt_{2}^{2}}
\bigg)
-
{
1-r^{4}
\over
8r^{2}
}
{\dd^{2}
\over
\dd \gt_{1}\dd\gt_{2}}
,
\eea
and the ladder operators are
\be\label{RisR}
\cR
=
e^{2 i \gt _2}
\bigg(
-i 
\frac{1-r^2}{2 r}
{\dd\over \dd \gt_{1}}
+
(1-r^{2})
{\dd\over \dd r}
+
i
\frac{1+r^2}{2 r}
{\dd\over \dd \gt_{2}}
\bigg)
,
\ee
and
$$
\cL
=
e^{-2 i \gt _2}
\bigg(
i 
\frac{1-r^2}{2 r}
{\dd\over \dd \gt_{1}}
+
(1-r^{2})
{\dd\over \dd r}
-
i
\frac{1+r^2}{2 r}
{\dd\over \dd \gt_{2}}
\bigg)
.
$$


\subsection{$K$-isotypic Vectors in the Line Model}
Turning now to the line model, let $\cH=V_{s}$ with grading as in \eqref{eq:grading}. 
\begin{lem}
In the line model, if $f_{2k,s}\in V_{s}^{(2k)}$ then
\be\label{eq:f2ks}
f_{2k,s}(x)
=
c
(x-i)^{k-s}
(x+i)^{-k-s}
\ee
for some $c\in\C$.
\end{lem}
\pf
For any vector $f\in V_{s}^{(2k)}$, the action of $Y=e-f\in\fg$ is $Y.v={\dd\over \dd\gt_{2}}v=2ikv$.
We compute in the line model that the element $Y$ acts on $V_{s}$ by:
$$
Y.f(x) = 2sxf(x)+(1+x^{2})f'(x).
$$
Hence $f\in V_{s}^{(2k)}$, if $f$ satisfies 
$$
 2sxf(x)+(1+x^{2})f'(x) = 2i k f(x)
.
$$
It is elementary to verify that \eqref{eq:f2ks} is the unique solution up to constant.
\epf

\subsection{Fourier Expansions}\label{secFEs}
\

Let $v_{2k}\in\cH^{(2k)}$ be a $K$-isotypic vector, and recall that the Casimir operator acts as  $\foh D(\cC)+\gl=0$, with $\gl=s(1-s)$ and $s>1/2$. Then as a function in $(\gt_{1},r,\gt_{2})$ coordinates, one has the Fourier expansion:
$$
v_{2k}(\gt_{1},r,\gt_{2})
= 
e^{2ik\gt_{2}}
\sum_{n\in\Z} 
v_{2n,2k}(r) \ 
e^{2in\gt_{1}}
.
$$
(Only even frequencies appear, since the representation factors through $\PSL(2,\R)$.) The differential equation induced from \eqref{CisR} on $v_{2n,2k}$ is:
\bea\label{CisNK}
&&
{
(1-r^{2})^{2}
\over 
4
}
{\dd^{2}
\over
\dd r^{2}}
v_{2n,2k}(r)
+
{
(1-r^{2})^{2}
\over 
4r
}
{\dd
\over
\dd r}
v_{2n,2k}(r)
\\
\nonumber
&&
\qquad
+
\bigg(
-
{
(1-r^{2})^{2}
\over
4 r^{2}
}
\bigg(
n^{2}
+
k^{2}
\bigg)
+
{
1-r^{4}
\over
2r^{2}
}
nk
+
s(1-s)
\bigg)
v_{2n,2k}(r)
=
0
.
\eea

Before solving this equation, we note that $v_{2k}$ is regular everywhere, in particular at the origin, and hence so is $v_{2n,2k}$. Expansion in series of \eqref{CisNK} about the origin $r=0$ gives
$$
(
1+O(r)
)
{\dd^{2}
\over 
\dd r^{2}}
+
{
1+O(r)
\over 
r
}
{\dd
\over
\dd r}
-
{
1+O(r)
\over
r^{2}
}
(n-k)^{2}
=
0
,
$$
which has asymptotic solution of the form:
$$
\twocase{}
{( c_{1} r^{n-k} + c_{2} r^{k-n} )(1+O(r))}{if $n\neq k$,}
{( c_{1}  + c_{2} \log r )(1+O(r))}{if $n= k$.}
$$
In the case when $n>k$ (respectively, $n<k$),  regularity at the origin forces   $c_{2}$ (respectively, $c_{1}$) to vanish.
If $n=k$, then $c_{2}=0$. Either way, there is a multiplicity one principle, and $v_{2n,2k}$ is a constant multiple of the unique solution  $\Phi_{2n,2k}$ to \eqref{CisNK}  having $1$ in its first non-zero Taylor coefficient in $r$ (note that $\Phi_{2n,2k}$ vanishes to order $|n-k|$ at the origin $r=0$). One can explicitly compute:
\be\label{PhiIs}
\Phi_{2n,2k}(r)
=
\left(1-r^2\right)^{s } r^{|n-k|}
{
 \, _2F_1\left(s-\e_{n,k}\, k,s
   +\e_{n,k}\,n;1+|n-k|;r^2\right)
}
,
\ee
where 
\be\label{enk}
\e_{n,k}=\twocase{}{1}{if $n\ge k$}{-1}{otherwise,}
\ee 
and $_{2}F_{1}$ is the standard Gauss hypergeometric series.
We will 
use the same name $\Phi_{2n,2k}$ for the function on $G$ defined by
$$
\Phi_{2n,2k}(k_{\gt_{1}}a_{t}k_{\gt_{2}})
=
e^{2in\gt_{1}}\
\Phi_{2n,2k}(r)\
e^{2ik\gt_{2}}
,
$$
where $r$ is related to $t$ by \eqref{rTot}.

\subsection{Choice of a Basis from Ladder Operators}
\

Take a $K$-fixed vector $v_{0}\in\cH$ (it is unique up to scalar), and assume that it has unit norm under the inner product $\<\cdot,\cdot\>$ in $\cH$ (whence it is unique up to a scalar of norm one). 
The raising (respectively, lowering) operator takes vectors in $\cH^{(2k)}$ to ones in $\cH^{(2k+2)}$ (respectively, $\cH^{(2k-2)}$), but the images no longer have unit norm. The normalization is as follows.

\begin{lem}

Let $\cX$ be a ladder operator, $\cX=\cR$ or $\cX=\cL$. Then for any $k\ge0$,
\be\label{bksDef}
\<
\cX^{k}.v_{0},
\cX^{k}.v_{0}
\>
=
2^{2 k}
{
\G(s+k)
\G(1-s+k)
\over 
\G(s)
\G(1-s)
}
=:
b_{k,s}
.
\ee
\end{lem}

\pf
We will exhibit the computation for $\cX=\cR$, the  case of $\cX=\cL$ being similar.
A standard calculation (recall that $Y=e-f$ in the basis \eqref{hef} and acts as ${\dd\over \dd\gt_{2}}$) shows that 
\be\label{LR}
\cL\cR
=
2\cC
+Y^{2}+2iY,
\ee
and hence acts on $\cH^{(2k)}$ as scalar multiplication by 
$$
-4 \gl +(2ik)^{2}+2i(2ik)
=
-4(s+k)(1-s+k) 
.
$$
Using
\be\label{Adjoint}
\<\cR.v,w\>
=
-
\<v,\cL.w\>
\ee
then gives
\beann
\<
\cR^{k}.v_{0},
\cR^{k}.v_{0}
\>
&=&
(-1)^{k}
\<
\cL^{k}\cR^{k}.v_{0},
v_{0}
\>
\\
&
=&
(-1)^{k}
(-4(s)(1-s))
(-4(s+1)(1-s+1))
\cdots
\\
&&
\qquad\qquad
\times
(-4(s+k-1)(1-s+k-1))
\<
v_{0},
v_{0}
\>
\\
&
=&
(-1)^{k}
(-1)^{k}
4^{k}
{
\G(s+k)
\G(1-s+k)
\over 
\G(s)
\G(1-s)
}
,
\eeann
as claimed.
\epf

Now given a fixed vector $v_{0}\in\cH^{(0)}$ of unit norm, we make once and for all  the following choice for an orthonormal basis 
for the space $\cH^{\infty}$ of smooth vectors.
\begin{Def}
{\rm
For $k\neq0$,
let
\be\label{v2kDef}
v_{2k}:=
\frac{
1}{
\sqrt{
b_{|k|,s}}
}
\twocase{}
{\cR^{k}.v_{0}}
{if $k>0$,}
{\cL^{|k|}.v_{0}}{if $k<0$,}
\ee
where $b_{k,s}$ is defined in \eqref{bksDef}.
}
\end{Def}

\

\subsection{Basis for $V_{s}$ in the Line Model}

In the line model $V_{s}$, we have the un-normalized functions $f_{2k,s}$ given in \eqref{eq:f2ks}. Hence to determine the relationship between $f_{2k,s}$ and $v_{2k}$ in \eqref{v2kDef}, we need only to normalize $f_{2k,s}$. First we compute the action of the intertwining operator.
\begin{lem}
For $f_{2k,s}$ as in \eqref{eq:f2ks}, we have
\be\label{eq:cINorm}
\cI.f_{2k,s} 
= 
{4^{1-s}\pi (-1)^{k} \G(2s-1)\over \G(s-k)\G(s+k)} 
f_{2k,1-s}
.
\ee
\end{lem}
\pf
The identity to be verified is
$$
\int_{\R}{(y-i)^{k-s}(y+i)^{-k-s}\over |x-y|^{2(1-s)}}dy 
=
{4^{1-s}\pi (-1)^{k} \G(2s-1)\over \G(s-k)\G(s+k)} 
(x-i)^{k-(1-s)}(x+i)^{-k-(1-s)}
.
$$
The intertwining operator preserves the group action, so $\cI.f_{2k,s}\in V_{1-s}^{(2k)}$, and hence is a multiple of $f_{2k,1-s}$. To determine the multiple, we may simply set $x=0$ in the above and compute the integral.
\epf
With this computation at hand, we may determine the norms of $f_{2k,s}$.
\begin{lem}
For $f_{2k,s}$ as in \eqref{eq:f2ks}, we have
\be\label{eq:f2ksNorm}
\<f_{2k,s},f_{2k,s}\>
=
{4^{1-s}\pi^{2} (-1)^{k} \G(2s-1)\over \G(s-k)\G(s+k)} 
=:
\tilde b_{k,s}>0
.
\ee
\end{lem}
\pf
By definition, the left hand side of \eqref{eq:f2ksNorm} is
\beann
&=&
\int_{\R}f_{2k,s}(x)\overline{\cI.f_{2k,s}(x)} dx
\\
&=&
{4^{1-s}\pi (-1)^{k} \G(2s-1)\over \G(s-k)\G(s+k)} 
\int_{\R}
(x-i)^{k-s}(x+i)^{-k-s}
\overline{
(x-i)^{k-(1-s)}(x+i)^{-k-(1-s)}
} dx
\\
&=&
{4^{1-s}\pi (-1)^{k} \G(2s-1)\over \G(s-k)\G(s+k)} 
\int_{\R}
(x^{2}+1)^{-1}
 dx
\\
&=&
{4^{1-s}\pi (-1)^{k} \G(2s-1)\over \G(s-k)\G(s+k)} 
\pi
,
\eeann
using \eqref{eq:cINorm}. Note that this value is real and positive.
\epf

\subsection{Ladder Operators on Fourier Expansions}
\

Fix some $v_{0}\in\cH$ and let $v_{2k}$ be the basis defined by \eqref{v2kDef}. From \S\ref{secFEs}, $v_{0}$ in coordinates $(\gt_{1},r ,\gt_{2})$ has Fourier expansion:
\be\label{v0FE}
v_{0}(\gt_{1},r,\gt_{2})
=
\sum_{n}
c_{2n}
\
\Phi_{2n,0}(r) \
e^{2in\gt_{1}}
,
\ee
with some Fourier coefficients $c_{2n}\in\C$.
In this subsection, we express the Fourier expansions of all the $v_{2k}$ in terms of the coefficients $c_{2n}$. For this we require the following

\begin{lem}\label{lemRLPhi}
The ladder operators act on $\Phi_{2n,2k}$ by:
\beann
\cR.\Phi_{2n,2k}
 &=&
-
2\
 \Phi_{2n,2k+2}
 \times  
\twocase{}
{-(n-k)}
{if $n>k$,}
{
 \dfrac{(s+k)(1-s+k) }{1-n+k}
}
{if $n\le k$,}
\eeann
and
\beann
\cL.\Phi_{2n,2k}
 &=&
-
2\
 \Phi_{2n,2k-2}
 \times  
\twocase{}
{
 \dfrac{(s-k)(1-s-k) }{1+n-k}
}
{if $n\ge k$,}
{-(k-n)}
{if $n< k$.}
\eeann
\end{lem}

\pf
We will demonstrate the case of acting by $\cR$ with $n>k$, the other cases being similar.
From \eqref{RisR}, we have that $\cR$ acts on $\Phi_{2n,2k}$ by
$$ 
\frac{1-r^2}{r}
n
+
(1-r^{2})
{\dd\over \dd r}
-
\frac{1+r^2}{r}
k
$$
Recall from \eqref{PhiIs} and $n>k$ that
$$
\Phi_{2n,2k}(r)
=
\left(1-r^2\right)^s r^{n-k} \, _2F_1\left(s-k,s+n;1-k+n;r^2\right),
$$
and hence a computation yields
\beann
\cR. \Phi_{2n,2k}(r)
&=&
-2 \left(1-r^2\right)^s r^{n-k-1} 
\\
&&
\quad\times\Bigg\{\left(r^2(k+
   n)+s-n\right) \,
   _2F_1\left(s-k,s+n;1-k+n;r^2\right)
\\
   &&
   \qquad
-\left(1-r^2\right) (s-k) \,
   _2F_1\left(s-k+1,s+n;1-k+n;r^2\right)\Bigg\}
.
\eeann
Using the series expansion of the Gauss hypergeometric series, it is a matter of combinatorics to verify that the above expression is the same as 
$$
2(n-k)
\left(1-r^2\right)^s r^{n-k-1} \, _2F_1\left(s-k-1,s+n;n-k;r^2\right),
$$
as claimed.
\epf

From Lemma \ref{lemRLPhi}, it follows after a calculation that for $k\ge0$,
\beann
\cR^{k}.v_{0}(\gt_{1},r,\gt_{2})
\hskip4in
\\
=
(-1)^{k}2^{k}e^{2ik\gt_{2}}
\sum_{n}c_{2n} \  \Phi_{2n,2k}(r) \ e^{2in\gt_{1}}
\hskip2in
\\
\qquad
\times
\threecase
{
(-1)^{k}
\dfrac{\G(n+1)}{\G(n-k+1)}
}
{if $n\ge k$,}
{
(-1)^{n}
\dfrac{
\G(n+1)\G(s+k)\G(1-s+k)
}{
\G(k-n+1)\G(s+n)\G(1-s+n)
}
}
{if $1\le n \le k-1$,}
{
\dfrac{
\G(|n|+1)\G(s+k)\G(1-s+k)
}{
\G(k+|n|+1)\G(s)\G(1-s)
}
}
{if $n\le 0$.}
\eeann
A similar identity holds for $\cL^{k}.v_{0}$. Recall that $\Phi_{2n,2k}$ vanishes at the origin $r=0$ unless $n=k$, in which case it takes the value $1$. We have proved
\begin{prop}
For $k\ge0$, the value at the origin  of $v_{0}$ acted on by ladder operators is related to its Fourier coefficients by
\be\label{RkOrig}
\cR^{k}.v_{0}(e)
=
c_{2k}\
2^{k}\
\G(k+1)
\ee
and
$$
\cL^{k}.v_{0}(e)
=
c_{-2k}\
2^{k}\
\G(k+1)
.
$$
\end{prop}
\


\subsection{Matrix Coefficients}
\

Fix integers $n$ and $k$. 
We first record here the asymptotic growth rate of $\Phi_{2n,2k}$ at infinity. 
\begin{lem}
As $t\to \infty$,
\be\label{MtrxAtInf}
\Phi_{2n,2k}(a_{t}) 
=
4^{1-s}\ e^{-t(1-s)}
{\G(1+|n-k|)\G(2s-1)\over \G(s-\e_{n,k}\, k)\G(s  +\e_{n,k}\, n)} 
(1+O(nk\,e^{-t}))
,
\ee
with $\e_{n,k}$ defined in \eqref{enk}.
\end{lem}
\pf
Recall the well-known identity
\bea
\nonumber
&&
\hskip-.5in
\ _{2}F_{1}(a,b,c;z)
\\
\label{2F1Id}
&=& 
{\G(c)\G(c-a-b)\over \G(c-a)\G(c-b)} \ _{2}F_{1}(a,b,a+b-c+1;1-z)
\\
\nonumber
&&
+
(1-z)^{c-a-b}
{\G(c)\G(a+b-c)\over \G(a)\G(b)} \ _{2}F_{1}(c-a,c-b,c-a-b+1;1-z)
.
\eea
Applied to the 
series in \eqref{PhiIs}  for the case $n\ge k$
, \eqref{2F1Id} gives
\bea\label{eq:2F1Trans}
&&
\hskip-.3in
_2F_1\left(s-k,s  +n,1-k+n;r^2\right)
   \\
   \nonumber
&=&
{\G(1-k+n)\G(1-2s)\over \G(1-s+n)\G(1-s-k)} \ _{2}F_{1}(s-k,s  +n,2s;1-r^2)
\\
   \nonumber
&&
+
(1-r^2)^{1-2s}
{\G(1-k+n)\G(2s-1)\over \G(s-k)\G(s  +n)} \ _{2}F_{1}(1-s+n,1-s-k,2-2s;1-r^2)
.
\eea
Using (cf. Good \cite{Good1983})
$$
_{2}F_{1}(a,b,c;z) = 1+O\left(\left|\frac {abz}c\right|\right)\qquad\text{for}\qquad|z|\max_{\ell\ge0}\left|\frac{(a+\ell)(b+\ell)}{(c+\ell)(\ell+1)}\right|\le\foh,
$$
the second term in \eqref{eq:2F1Trans} is the only one growing as $r\to 1$, and hence
$$
\Phi_{2n,2k}(r)
=
(1-r^{2})^{s}r^{n-k}  (1-r^2)^{1-2s}
{\G(1-k+n)\G(2s-1)\over \G(s-k)\G(s  +n)} (1+O(nk(1-r^{2})))
.
$$
Simplifying and changing variables \eqref{rTot} gives \eqref{MtrxAtInf}. The case $n<k$ is similar.
\epf
\

Let $\pi$ denote the right-regular representation on the irreducible $\cH$. Take the $K$-isotypic vectors $v_{2n}$ and $v_{2k}$ in the basis \eqref{v2kDef}, and form the matrix coeffcient:
$$
M_{2n,2k}(g):=\<\pi(g).v_{2k},v_{2n}\>.
$$
Note that $M_{2n,2k}$ is an eigenfunction of the Casimir operator, and 
transforms
by
$$
M_{2n,2k}(k_{\gt_{1}}\, g\, k_{\gt_{2}}) 
=
e^{2in\gt_{1}}
M_{2n,2k}(g) 
e^{2ik\gt_{2}}
,
$$
whence is a scalar multiple of $\Phi_{2n,2k}$.
For instance, if  $n=k$, then we instantly have $M_{2n,2n}(g)=\Phi_{2n,2n}(g).$
In the sequel, we require knowledge of this constant in the general case.
\begin{lem}
For integers $n$ and $k$, 
\bea
\label{MatCoef}
\<\pi(g).v_{2k},v_{2n}\>
&
=
&
\Phi_{2n,2k}(g)
\left(
\tilde b_{k,s} \ 
\tilde b_{n,s}
\right)^{-1/2}
\\
\nonumber
&&
\times
{(-1)^{k}
4^{1-s}\pi^{2}  \G(2s-1)
\over 
\G(1+|n-k|)
\G(s-\e_{n,k}\,n)
\G(s+\e_{n,k}\,k)
}
.
\eea
\end{lem}

\pf
We carry out this computation by switching to the line model. For $g=a_{t}$, and relating $t$ to $r$ via \eqref{rTot}, the left hand side of \eqref{MatCoef} is 
\beann
&=&
(\tilde b_{k,s} \ \tilde b_{n,s})^{-1/2}
{
\<\pi(a_{t}).f_{2k,s},f_{2n,s}\>
}
\\
&=&
(\tilde b_{k,s} \ \tilde b_{n,s})^{-1/2}
\int_{\R} 
\left({1+r\over 1-r}\right)
^{s}
\left(
\left({1+r\over 1-r}\right)
x-i
\right)^{k-s}
\left(
\left({1+r\over 1-r}\right)
x+i
\right)^{-k-s}
\\
&&
\hskip1.3in
\times
{4^{1-s}\pi (-1)^{n} \G(2s-1)\over \G(s-n)\G(s+n)} 
\overline{
(x-i)^{n-(1-s)}
(x+i)^{-n-(1-s)}
}
dx
,
\eeann
using \eqref{eq:cINorm}.
Qua a function of $r$, the above integral is a multiple of \eqref{PhiIs}. To determine the multiple, we may simply study the asymptotics at infinity, corresponding to $r\to1$, using Laplace's method, and compare it to \eqref{MtrxAtInf}. 

In this way, one obtains 
\beann
&&
\int_{\R} 
\left({1+r\over 1-r}\right)
^{s}
\left(
\left({1+r\over 1-r}\right)
x-i
\right)^{k-s}
\left(
\left({1+r\over 1-r}\right)
x+i
\right)^{-k-s}
\\
&&
\hskip1.5in
\times
\overline{
(x-i)^{n-(1-s)}
(x+i)^{-n-(1-s)}
}
dx
\\
&&
\hskip2in
=
{
(-1)^{k+n}\pi
\over
\G(1+|n-k|)
}
{
\G(s+\e_{n,k}\,n)
\over
\G(s+\e_{n,k}\,k)
}
\Phi_{2k,2n}(r)
.
\eeann
Combining the constants completes the proof.
\epf

\

\subsection{The Patterson Sullivan measure}
\

Recall  from \eqref{muPS} the Patterson-Sullivan measure $\mu$, supported on the limit set $\gL\subset\dd\bH$, which has Hausdorff dimension $\gd$ with $1/2<\gd<1$. The  eigenfunction 
$\varphi_{0}$ corresponding to the base eigenvalue $\gl_{0}=\gd(1-\gd)$ is expressed explicitly in disk coordinates $(\gt_{1},r,\gt_{2})$ as the integral of a Poisson kernel against $\mu$ as follows:
\be\label{rep1}
\varphi_{0}(\gt_{1},r,\gt_{2})
=
\int_{0}^{\pi}
\left(
{
1-r^{2}
\over
|re^{2i\gt_{1}}-e^{2i\ga}|^{2}
}
\right)^{\gd}
d\mu(\ga)
.
\ee
Recall from \eqref{v0FE} that  $\varphi_{0}$ has Fourier development:
\be\label{rep2}
\varphi_{0}(\gt_{1},r,\gt_{2})=\sum_{n}c_{2n}\ \Phi_{2n,0}(r)\ e^{2in\gt_{1}}
.
\ee
The $2n$-th Fourier cofficient of $\mu$ is given by:
$$
\hat\mu(2n)= \int_{0}^{\pi} e^{2in\ga}d\mu(\ga)
.
$$

\begin{lem}
The relationship between the coefficient $c_{2n}$ and $\hat\mu$
is given explicitly by
\be\label{c2nMu}
c_{2n}
=
{1\over\G(\gd)}{\G(\gd+|n|)\over \G(1+|n|)}\ {\hat{\mu}(-2n)}
.
\ee
\end{lem}
\pf
Equating the two expressions \eqref{rep1} and \eqref{rep2}, inserting \eqref{PhiIs} with $s=\gd$, and dividing both sides by $(1-r^{2})^{\gd}$ gives
$$
\sum_{n}c_{2n}\ r^{|n|} \ _{2}F_{1}(\gd,\gd+|n|,1+|n|,r^{2})e^{2in\gt_{1}}
=
\int_{0}^{\pi}
(r^{2}-r(e^{i(2\gt_{1}-2\ga)}+e^{i(2\ga-2\gt_{1})})+1
)^{-\gd}
d\mu(\ga)
.
$$
Expanding both sides in  series and collecting terms yields \eqref{c2nMu}.
\epf
\

\subsection{Decay of Tempered Matrix Coefficients}
\

We end this section by  recalling  the well-known strong mixing property for matrix coefficients \cite{HoweMoore1979, Cowling1978, CowlingHaagerupHowe1988, Venkatesh2008}.

\begin{lem}
Let $(\pi,V)$ be a tempered unitary representation of $G
$.
Then for any vectors $v,w\in V$ whose $K$-span is one-dimensional, 
\be\label{decayMtrx}
|\<\pi(k_{\gt_{1}}a_{t}\,k_{\gt_{2}}).v,w\>|
\ll
t e^{- t/2}\,
\|v\|_{2}\|w\|_{2}
,
\text{ as $t\to\infty$,}
\ee
where implied constant is absolute.
\end{lem}

Define the Sobolev norm 
$$  
\cS v=\|v\|_{2} + \| d\pi(h). v \|_2 + \| d\pi(e). v \|_2 +\| d\pi(f).v \|_2,
$$
where $h$, $e$, and $f$ are an orthonormal basis for $\fg$, cf. \eqref{hef}.

\begin{lem}
Let $\gT>1/2$ and $(\pi,V)$ be a unitary representation of $G
$ which does not weakly contain any complementary series representation with parameter $s>\gT$. 
Then for any smooth vectors $v,w\in V^{\infty}$, 
\be\label{decayMtrxgT}
|\<\pi(k_{\gt_{1}}a_{t}\,k_{\gt_{2}}).v,w\>|
\ll
e^{- \gT t}\,
(\|v\|\|w\|)^{1/2}
(\cS v \, 
\cS w)^{1/2}
,
\text{ as $t\to\infty$,}
\ee
where implied constant is absolute.
\end{lem}
\

\section{Proof of Theorem \ref{thm:main}}\label{sec3}

Let $\G<G=\SL_2(\R)$
be a Fuchsian group of the second kind, and let $K=\SO(2)$ be the maximal compact.
Assume the limit set of $\G$ has Hausdorff dimension $\gd>1/2$. 
Let 
$$
0<\gd(1-\gd)=\gl_0<\gl_1\le\cdots\le\gl_N<1/4
$$ 
be the point spectrum of the Laplacian acting on $L^2(\G\bk \bH)$, with 
$$
\gl_j=s_j(1-s_j)
$$ 
and
$s_{j}>1/2$.
\\

Fix integers $n$ and $k$. Our goal is to evaluate
$$
\cN
(T)
:=
\sum_{\g\in\G\atop\|\g\|<T}
e^{2in\,\gt_{1}(\g)}
\ 
e^{2ik\,\gt_{2}(\g)}
.
$$

For $g\in G$, let
$$
f_T
(g):=
e^{2in\,\gt_{1}(g)}
\ 
e^{2ik\,\gt_{2}(g)}
\
{\bf1}_{\|g\|<T}
,
$$
and define $\cF_T:\G\bk G\times\G\bk G\to\C$ via
\be\label{cFTdef}
\cF_T(g,h)
:=
\sum_{\g\in\G}f_T(g^{-1}\g h)
.
\ee
Clearly $\cF_{T}(e,e)=\cN(T)$.
\\

For a fixed $\eta$ (to be chosen later depending on $T$), let 
$$
\psi
:
\G\bk G
\to\R
$$
be a smooth test function with
unit mass,
 $\int_{\G\bk G}\psi=1$,
and compact
support in 
a ball of radius $\eta$ about the identity $e\in G$.
Then the integral
\be\label{cHTdef}
\cH
(T):=\<\cF_T,\psi
\otimes \psi
\>
=
\int_{\G\bk G}\int_{\G\bk G}\cF_{T}(g,h)\psi(g)\psi(h)dg\, dh
\ee
approximates $\cN(T)$ 
as follows.

\begin{prop}
\be\label{propH1}
\cH(T)=\cN(T)
+
O
\left(
\eta(1+|n|+|k|)
T^{2\gd}
\right)
.
\ee
\end{prop}
\begin{proof}
Writing 
$$
\cF_{T}(g,h)=\cF_{T}(e,e) + (\cF_{T}(g,h)-\cF_{T}(e,e))
$$ 
and using $\int\psi=1$ gives
$$
\cH(T)=\cN(T)  + \cE(T),
$$
where
\beann
\cE(T)
&: =& 
\int_{\G\bk G}\int_{\G \bk G} (\cF_{T}(g,h)-\cF_{T}(e,e))\psi(g)\psi(h)dg\, dh
\\
&=&
\sum_{\g\in\G}
\int_{\G\bk G}\int_{\G \bk G} 
(f_{T}(g^{-1}\g h)-f_{T}(\g))
\psi(g)\psi(h)dg\, dh
.
\eeann
Recall that $\psi$ has support in a ball of radius $\eta$ about the identity, and let $g,h\in \supp\psi$. 
For $\g\in\G$, we consider three ranges of $\|\g\|$ separately:
\ben
\item If $\|\g\|>
{T\over 1-\eta},$ then both $f_{T}(g^{-1}\g h)$ and $f_{T}(\g)$ vanish. 
\item If ${T\over 1+\eta}<
\|\g\|\le
{T\over 1-\eta},$ then we estimate trivially
\be\label{Est2}
\left| 
f_{T}(g^{-1}\g h)-f_{T}(\g)
\right|
\le 2
.
\ee
\item Lastly, if $\|\g\|\le
 {T\over 1+\eta},$ then 
$$
{\bf 1}_{\|g^{-1}\g h\|<T}
=
{\bf 1}_{\|\g \|<T}
=
1
,
$$
and from
$$
|
e^{2in\,\gt_{1}(g^{-1}\g h)}
-
e^{2in\,\gt_{1}(\g)}
| \ll |n| \eta
$$
(using $KA^{+}K$ coordinates),
it follows that
\bea
\label{Est3}
\left| 
f_{T}(g^{-1}\g h)-f_{T}(\g)
\right|
&\ll&
(|n|+|k|)\eta
.
\eea
\een
Combining \eqref{Est2} and \eqref{Est3} gives
\beann
\cE(T)
&\ll&
\sum_{\g\in\G\atop{T\over 1+\eta}<
\|\g\|\le
{T\over 1-\eta} }1
\quad
+
\quad
(|n|+|k|)\eta
\sum_{\g\in\G\atop \|\g\|\le
{T\over 1+\eta} }1
\\
&\ll&
\eta T^{2\gd}
+
(|n|+|k|)\eta
T^{2\gd}
,
\eeann
by Lax-Phillips \cite{LaxPhillips1982}. This completes the proof.
\end{proof}


It remains to
evaluate $\cH(T)$. First we rewrite it, as follows.

\begin{lem}
The inner product $\cH(T)$ can be expressed as the follows: 
\be\label{HTmatCoeff}
\cH(T) 
= 
\int_{G}f_{T}(g)\<\pi(g).\psi,\psi\> dg
.
\ee
\end{lem}
\begin{proof}
Insert the definition of $\cF_{T}$ \eqref{cFTdef} into \eqref{cHTdef}, interchange summation and integration,  change variables $g=x^{-1}\g h$, and use the left $\G$-invariance of $\psi$:
\beann
\cH(T)
&=&
\int_{x\in\G\bk G}
\int_{h\in\G\bk G}
\sum_{\g\in\G}f_T(x^{-1}\g h)
\psi(h)
\psi(x)
dh\, dx
\\
&=&
\int_{x\in\G\bk G}
\left(
\sum_{\g\in\G}
\int_{h\in\G\bk G}
f_T(x^{-1}\g h)
\psi(h)
dh
\right)
\psi(x)
\, dx
\\
&=&
\int_{x\in\G\bk G}
\left(
\sum_{\g\in\G}
\int_{g\in\g^{-1}x(\G\bk G)}
f_T(g)
\psi(\g^{-1}xg)
dg
\right)
\psi(x)
\, dx
\\
&=&
\int_{x\in\G\bk G}
\left(
\int_{g\in G}
f_T(g)
\psi(xg)
dg
\right)
\psi(x)
\, dx
,
\eeann
since for $x$ fixed,  $x(\G\bk G)$ is a fundamental domain for $\G$, and hence
$\sum_{\g\in\G}\int_{\g^{-1}x(\G\bk G)}
=\int_{G}$.
Interchanging integrals gives
$$
\cH(T)
=
\int_{G}
f_T(g)
\int_{\G\bk G}
\psi(xg)
\psi(x)
dx
\, 
dg
,
$$
as desired.
\end{proof}

At this point, we could expand the matrix coefficient $\<\pi(g).\psi,\psi\>$ spectrally, but the error term would then contain more harmonics than necessary, leading to worse bounds (essentially requiring the Sobolev norms arising in \eqref{decayMtrxgT}, in place of the $\cL^{2}$ norms in \eqref{decayMtrx}). So we first remove the immaterial harmonics.
To this end, decompose $\psi$ into its Fourier series with respect to $\gt_{2}$,
\be\label{psiFE}
\psi(g)=\sum_{m\in\Z} \psi_{2m}(g),
\ee
where $\psi_{2m}$ transforms on the right by
$$
\psi_{2m}(g\,k_{\gt})=\psi_{2m}(g)e^{2im\gt}.
$$
Insert \eqref{psiFE} twice into \eqref{HTmatCoeff}:
$$
\cH(T)=\sum_{m}\sum_{\ell}
\int_{G}f_{T}(g)
\<\pi(g).\psi_{2m},\psi_{2\ell}\>
 dg
.
$$

Note that   the matrix coefficient above transforms on the left and right by
$$
\<\pi(k_{\gt_{1}}g\,k_{\gt_{2}}).\psi_{2m},\psi_{2\ell}\>
=
e^{2i\ell\gt_{1}}
\<\pi(g).\psi_{2m},\psi_{2\ell}\>
e^{2im\gt_{2}}
,
$$
whence
$$
\int_{G}f_{T}(g)
\<\pi(g).\psi_{2m},\psi_{2\ell}\>
dg
=0
,
$$
unless 
 $m=-k$ and $\ell=-n$. Having dispensed with extraneous harmonics, we write
\be\label{Hnow}
\cH(T)=
\int_{G}f_{T}(g)
\<\pi(g).\psi_{-2k},\psi_{-2n}\>
 dg
,
\ee
and now expand the matrix coefficient spectrally. Recall that $\gl_{0}=\gd(1-\gd)$ is the base frequency with corresponding eigenfunction $\varphi_{0}$, and  assume at first that it is the sole discrete eigenvalue, the rest of the spectrum being tempered. Let $V$ denote the vector space consisting of the closure of the $G$-span of $\vf_{0}$, and use the notation $v_{0}=\varphi_{0}$ and \eqref{v2kDef}.

As the matrix coefficient in \eqref{Hnow} is bi-$K$-isotypic, only one mode is excited in each expansion.
Hence
\be\label{pipsiExp}
\<\pi(g).\psi_{-2k},\psi_{-2n}\>
=
\<\psi_{-2k},v_{-2k}\>
\<v_{-2n},\psi_{-2n}\>
\<\pi(g).v_{-2k},v_{-2n}\>
+
\<\pi(g).\psi_{-2k}^{\perp},\psi_{-2n}^{\perp}\> 
,
\ee
where the $K$-spans of $\psi_{1}^{\perp}$ and $\psi_{2}^{\perp}$ are one-dimensional and the last matrix coefficient is tempered.
Note that
\be\label{psi2kpsi}
\<\psi_{-2k},v_{-2k}\> = 
\<\psi,v_{-2k}\>
.
\ee
\begin{prop}
As $T\to\infty$,
\bea
\nonumber
\hskip-.5in\cH(T)
&=&
\<\psi , v_{-2k}\>
\<v_{-2n},\psi \>
\int_{t=0}^{2\log T}
\<\pi(a_{t}).v_{-2k},v_{-2n}\>
\sinh t\
dt
\\
\label{cHTT}
&&
\hskip2in
+
O
\left(
\|\psi\|_{2}^{2}\
T\log T
\right)
.
\eea
\end{prop}
\pf
The main term is simply a combination of \eqref{Hnow}, \eqref{pipsiExp}, and \eqref{psi2kpsi}.
It remains to estimate
$$
\int_{G}f_{T}(g)\<\pi(g).\psi^{\perp}_{-2k},\psi^{\perp}_{-2n}\>.
$$
Take absolute values
and apply mixing \eqref{decayMtrx}. The Haar measure on $da_{t}$ is $\sinh t dt$,
giving
\be\label{applyMtxC}
\|\psi\|^{2}
\int_{A^{+}}\bo_{\|a_{t}\|<{T}}\ t\, e^{-t/2} \sinh t dt
\ll 
\|\psi\|_{2}^{2}\
T\log T
,
\ee
as claimed.
\epf

Recall that $\psi$
has unit mass, and is compactly supported in a ball of radius $\eta$ about the origin. This fact has two implications: the first is that
$$
\<v_{-2n},\psi\> = v_{-2n}(e) + O(\eta),
$$
and the second is that, since  $G$ is  a $3$-dimensional space,  we have
\be\label{psidim}
\|\psi\|^{2} \ll \eta^{-3}.
\ee
Combining these facts with \eqref{propH1}, we now have 
\bea
\label{HT3}
\cN(T)
&=&
\bar v_{-2k}(e)
v_{-2n}(e)
\int_{t=0}^{2\log T}
\<\pi(a_{t}).v_{-2k},v_{-2n}\>
\sinh t\
dt
\\
\nonumber
&&
\hskip1in
+
O
\bigg(
\eta(1+|n|+|k|)
T^{2\gd}
+
\eta^{-3}\
T\log T
\bigg)
.
\eea
The optimal choice 
$$
\eta=T^{(1-2\gd)/4}(\log T)^{1/4}(1+|n|+|k|)^{-1/4}
$$ 
in \eqref{HT3} leads to the error term 
$$
O\bigg(T^{1\cdot\frac14+2\gd\cdot\frac34}(\log T)^{1/4} (1+|n|+|k|)^{3/4}\bigg)
,
$$ 
as claimed.

Returning to \eqref{HT3}, it remains only to evaluate the main term. This is simply a matter of combining 
\eqref{RkOrig}, \eqref{v2kDef}, and \eqref{c2nMu}, together with \eqref{MatCoef} and \eqref{MtrxAtInf}. This completes the proof of Theorem \ref{thm:main}.

\

\

\section{Proof of Theorem  \ref{thm:unif}}\label{sec4}

As before, let $\G<G=\SL_2(\R)$
be a Fuchsian group of the second kind and assume the limit set of $\G$ has Hausdorff dimension $\gd>1/2$. 
Assume $\G<\SL(2,\Z)$ with ramification number $\fB$, and 
let $\gT$ be a spectral gap for $\G$.
For $q\ge1$, write $q=q'q''$ with $q'\mid\fB$, and 
let 
$$
0<\gd(1-\gd)=\gl_0^{q}<\gl_1^{q}\le\cdots\le\gl^{q}_{N(q)}<1/4
$$ 
be the point spectrum of the Laplacian acting on $L^2(\G(q)\bk \bH)$. The eigenvalues below $\gT(1-\gT)$ are  all oldforms coming from level $1$, with the possible exception of finitely many eigenvalues coming from level $q'$.
\\

Fix a function $f(g)=f(\gt_{1}(g),t(g),\gt_{2}(g))$ in  $KA^{+}K$ coordinates, and fix any $\g_{0}\in\G$. Assume that $|f|\le 1$. Our goal in this section is to evaluate
$$
\cN_{q}
(T)
:=
\sum_{\g\in\g_{0}\cdot\G(q)\atop\|\g\|<T}
f(\g)
.
$$

For $g\in G$, let
$$
f_T
(g):=
f(g)\
{\bf1}_{\|g\|<T}
,
$$
and define $\cF_{T,q}:\G(q)\bk G\times\G(q)\bk G\to\C$ via
\be\label{cFTdefq}
\cF_{T,q}(g,h)
:=
\sum_{\g\in\G(q)}f_T(g^{-1}\g h)
.
\ee
Clearly $\cF_{T,q}(\g_{0}^{-1},e)=\cN_{q}(T)$.
\\

Now for a fixed $\eta$ (to be chosen later depending on $T$), let 
$$
\psi
:
 G
\to\R
$$
be a smooth test function with
unit mass,
 $\int_{ G}\psi=1$,
and compact
support in 
a ball of radius $\eta$ about the identity $e\in G$.
Average over the group:
\be\label{PsiqDef}
\Psi_{q}(g):=\sum_{\g\in\G(q)}\psi(\g g).
\ee
Let 
$$
\psi_{\g_{0}}(g)=\psi(\g_{0}g),
$$
and similarly average 
$$
\Psi_{q,\g_{0}}(g)
:=
\Psi_{q}(\g_{0}g)
=
\sum_{\g\in\G(q)}\psi_{\g_{0}}(\g g)
.
$$
The integral
\bea
\nonumber
\cH_{q}
(T)&:=&\<\cF_T,\Psi_{q,\g_{0}}
\otimes \Psi_{q}
\>
=
\int_{\G(q)\bk G}\int_{\G(q)\bk G}\cF_{T}(g,h)\Psi_{q,\g_{0}}(g)\Psi_{q}(h)dg\, dh
\\
\label{cHTdefq}
&=&
\int_{\G(q)\bk G}\int_{\G(q)\bk G}\cF_{T}(\g_{0}^{-1}g,h)\Psi_{q}(g)\Psi_{q}(h)dg\, dh
\eea
again approximates $\cN_{q}(T)$ 
as follows.
Recall that 
$$
\cS_{\infty,T}f=\max_{X\in\{0,X_{1},X_{2},X_{3}\}}\sup_{g\in G, \|g\|<T}|d\pi(X).f(g)|
,
$$

\begin{prop}
\be\label{propH1q}
\cH_{q}(T)=\cN_{q}(T)
+
O
\left(
\eta(1+ \cS_{\infty,T}f)
T^{2\gd}
\right)
.
\ee
\end{prop}
\begin{proof}

This is the same argument as in the proof of Proposition \ref{propH1}, except 
using the bound
\be
\label{Est3q}
\left| 
f_{T}(g^{-1}\g_{0}\g h)-f_{T}(\g_{0}\g)
\right|
\ll
\eta\
\cS _{\infty,T} f
,
\ee
for $\|\g_{0}\g\|<T/(1+\eta)$.
\end{proof}

The argument leading to \eqref{HTmatCoeff} also gives

\begin{lem}
The inner product $\cH_{q}(T)$ can be expressed as: 
\be\label{HTmatCoeffq}
\cH_{q}(T) 
= 
\int_{G}f_{T}(g)\<\pi(g).\Psi_{q},\Psi_{q,\g_{0}}\>_{\G(q)\bk G} dg
.
\ee
\end{lem}

For ease of exposition, assume the spectrum below $\gT(1-\gT)$ consists of only the base eigenvalue $\gl_{0}=\gd(1-\gd)$ corresponding to $\vf^{(q)}$, and one newform $\tilde\vf^{(q)}$ from the ``bad'' level $q'\mid\fB$. The general case is a finite sum of such terms.
The normalizations are such that
\be\label{vfQnorm}
\vf^{(q)} = {1\over \sqrt{[\G:\G(q)]}}\vf^{(1)},
\ee
and
$$
\tilde\vf^{(q)} 
= 
{1\over \sqrt{[\G(q'):\G(q)]}}
\tilde
\vf^{(q')}
= 
{1\over 
\sqrt{[\G:\G(q)]}
}
\sqrt{[\G:\G(q')]}
\tilde
\vf^{(q')}
,
$$
with $\vf^{(1)}$ a normalized newform in $\cL^{2}(\G\bk G)$, and $\tilde\vf^{(q')}\in\cL^{2}(\G(q')\bk G)$. Let $V$ and $\tilde V$ be the irreducible vector subspaces of $\cL^{2}(\G(q)\bk G)$ generated by the $G$-spans of $\vf^{(q)}$ and $\tilde \vf^{(q)}$, respectively. The space $V$ has a dense subspace spanned by the $K$-fixed vector $\vf^{(q)}$ and its translates $\vf^{(q)}_{2k}$ under ladder operators, and similarly with $\tilde V$. 
Write
$$
\Psi_{q} 
= 
\Psi_{q}\big|_{V}
+
\Psi_{q}\big|_{\tilde V}
+
\Psi_{q}^{\perp}
,
$$
and similarly with $\Psi_{q,\g_{0}}$,
where the projections are
\be\label{PsiqProj}
\Psi_{q}\big|_{V}
:=
\sum_{k\in\Z}
\<\Psi_{q},\vf^{(q)}_{2k}\>
\vf_{2k}^{(q)},
\ee
etc.
Using \eqref{HTmatCoeffq}, we can now write
\be\label{cHqDecomp}
\cH_{q}(T)
=
W_{q}(T)
+
\tilde W_{q}(T)
+
W^{\perp}_{q}(T),
\ee
where
\be\label{WqTdef}
W_{q}(T)
:=
\int_{G}f_{T}(g)\<\pi(g).\Psi_{q}\big|_{V},\Psi_{q,\g_{0}}\big|_{V}\>_{\G(q)\bk G} dg
,
\ee
and similarly with the other two pieces.

\begin{lem}
\label{Qunfold}
$$
\<\Psi_{q},\vf_{2k}^{(q)}\>_{\G(q)\bk G}
=
{1\over \sqrt{[\G:\G(q)]}}
\<
\Psi_{1}\vf_{2k}^{(1)}\>_{\G\bk G}
,
$$
where 
$\Psi_{1}$ is defined by averaging over all of $\G$, as in \eqref{PsiqDef}. The same equality holds for $\Psi_{q,\g_{0}}$.
\end{lem}

\pf
Using \eqref{vfQnorm} and \eqref{PsiqDef}, unfold and refold the sum:
\beann
\<\Psi_{q},\vf_{2k}^{(q)}\>_{\G(q)\bk G}
&=&
{1\over \sqrt{[\G:\G(q)]}}
\int_{\G(q)\bk G}
\sum_{\g\in\G(q)}\psi(\g g)\vf_{2k}^{(1)}(g)dg
\\
&=&
{1\over \sqrt{[\G:\G(q)]}}
\int_{ G}
\psi( g)\vf_{2k}^{(1)}(g)dg
\\
&=&
{1\over \sqrt{[\G:\G(q)]}}
\int_{\G\bk G}
\sum_{\g\in\G}
\psi(\g g)\vf_{2k}^{(1)}(g)dg
\\
&=&
{1\over \sqrt{[\G:\G(q)]}}
\<
\Psi_{1}(g)\vf_{2k}^{(1)}\>_{\G\bk G}
,
\eeann
as claimed. Replacing $\Psi_{q}$ by $\Psi_{q,\g_{0}}$, one makes the additional change of variables $g\mapsto \g_{0} g$, and uses the $\G$-invariance of $\vf^{(1)}_{2k}$.
\epf

\begin{lem}\label{qMatCoef}
For any $k,k'\in\Z$,
$$
\<
\pi(g)\vf_{2k}^{(q)}
,
\vf_{2k'}^{(q)}
\>_{\G(q)\bk G}
=
\<
\pi(g)\vf_{2k}^{(1)}
,
\vf_{2k'}^{(1)}
\>_{\G\bk G}
.
$$
\end{lem}
\pf
Using \eqref{vfQnorm} on each function gives a factor of ${1\over[\G:\G(q)]}$, which 
is cancelled by the fact that the integral over the space $\G(q)\bk G$ is $[\G:\G(q)]$ times larger than that over $\G\bk G$.
\epf

\begin{lem}\label{WqTW1T}
$$
W_{q}(T)
=
{1\over [\G:\G(q)]}
W_{1}(T),
$$
where
$$
W_{1}(T)
:=
\int_{G}f_{T}(g)\<\pi(g).\Psi_{1}\big|_{V},\Psi_{1}\big|_{V}\>_{\G\bk G} dg
.
$$
\end{lem}
\pf
This is simply a combination of \eqref{PsiqProj}, \eqref{WqTdef}, and Lemmata  
 \ref{Qunfold}
and
\ref{qMatCoef}.
\epf

In the same way, one proves
\begin{lem}\label{tilWqT}
$$
\tilde W_{q}(T)
=
{1\over [\G:\G(q)]}
\cE_{q'}(T)
,
$$
where
$$
\cE_{q'}(T)
:=
[\G:\G(q')]
\int_{G}
f_{T}(g)
\<\pi(g).
\Psi_{q'}\big|_{\tilde V},
\Psi_{q',\tilde \g_{0}}\big|_{\tilde V}\>_{\G(q')\bk G}
dg.
$$
Here $\tilde \g_{0}$ is a representative for $\g_{0}$ in $\G(q')\bk \G$.
\end{lem}

The term $W_{q}^{\perp}$ is handled using the spectral gap in a similar way as 
\eqref{applyMtxC}. 
\begin{lem}\label{WqPerpT}
$$
W_{q}^{\perp}(T)
\ll
T^{2\gT}
\eta^{-6}
.
$$
\end{lem}

\pf
The bound   \eqref{decayMtrxgT} gives
$$
W_{q}^{\perp}(T)
\ll
T^{2\gT}\
\|\Psi_{q}\|_{2}\
 \cS\Psi_{q}
.
$$
We estimate
$$
\|\Psi_{q}\|
\ll
\eta^{-3/2}
,\qquad
\text{ and }
\qquad
\cS\Psi_{q} \ll
\eta^{-9/2}
,
$$
since $\Psi_{q}$ is a bump function in a $3$-dimensional ball of radius $\eta$.

\epf

Putting everything together gives
\begin{prop}
$$
\cN_{q}(T) =
{1\over [\G:\G(q)]}
\bigg(
\cN_{1}(T)
+
\cE_{q'}(T)
\bigg)
+
O\bigg(
T^{2\gT}\eta^{-6}
+\eta(1+\cS _{\infty,T} f)T^{2\gd}
\bigg)
$$
\end{prop}
\pf
Combining \eqref{propH1q}, \eqref{cHqDecomp}, and Lemmata
\ref{WqTW1T}, \ref{tilWqT},  and \ref{WqPerpT}, gives
$$
\cN_{q}(T) =
{1\over [\G:\G(q)]}
\bigg(
W_{1}(T)
+
\cE_{q'}(T)
\bigg)
+
O\bigg(
T^{2\gT}\eta^{-6}
+\eta(1+\cS _{\infty,T} f)T^{2\gd}
\bigg)
.
$$
On the other hand, the same argument gives
$$
\cN_{1}(T) =
W_{1}(T)
+
O\bigg(
T^{2\gT}\eta^{-6}
+\eta(1+\cS _{\infty,T} f)T^{2\gd}
\bigg)
.
$$
Combining
the two 
completes the proof.
\epf

The optimal choice of 
$$
\eta
=
(1+\cS _{\infty,T} f)^{-1/7}T^{-2(\gd-\gT)/7}
$$
gives the error term claimed in Theorem \ref{thm:unif}.

\

\

\section{Proof of Theorem \ref{thm:LowBnd}}\label{sec:LowBnd}

Assume $\G<\SL(2,\Z)$ has 
critical exponent
 $\gd>1/2$.
Recall that $N$ is a parameter going to infinity, $T$ and $K$ are small positive powers of $N$, $v,w\in\Z^{2}$, $n\in\Z$, $\frac NK<|n|<N$,
$|w|< {N\over T}$, $|v|\le 1$, and
$
|n|<|v||w|T.
$
We wish to give a lower bound for the number of $\g\in\G$, $\|\g\|<T$ such that
$$
\big|
\<v
\g,w\>
-n
\big|
<
{N\over K}
.
$$

Decompose $\g$ in $KA^{+}K$ coordinates, 
\beann
\g
&=&
k_{u}a_{\rho}k_{v}
=
\mattwo{\cos u}{\sin u}{-\sin u}{\cos u}
\mattwo{\rho}{}{}{1/\rho}
\mattwo{\cos v}{\sin v}{-\sin v}{\cos v}
\\
&=&
\left(
\begin{array}{cc}
 \sqrt{ } \cos (u) \cos (v)-\frac{\sin (u) \sin
   (v)}{{\rho }} & \frac{\cos (v) \sin
   (u)}{{\rho }}+{\rho } \cos (u) \sin
   (v) \\
 -{\rho } \cos (v) \sin (u)-\frac{\cos (u)
   \sin (v)}{{\rho }} & \frac{\cos (u) \cos
   (v)}{{\rho }}-{\rho } \sin (u) \sin
   (v)
\end{array}
\right)
.
\eeann 
As $1<\rho\approx\|\g\|$, 
we have
$$
\rho< T
.
$$

Let $v=(a,b)$ and $w=(c,d)$. 
The condition 
$$
|\<v\g,w\> - n| <{N\over K}
$$
becomes in $(u,\rho,v)$ coordinates
\beann
|\<v\g,w\> - n
|
&=&
\left|
(a,b)k_{u}a_{\rho}k_{v} \,^{t}(c,d) - n
\right|
\\
&=&
\Bigg|
\bigg(
\rho
(a\cos u - b \sin u)
,
\frac1{\rho}
(a\sin u + b \cos u)
\bigg)
\\
&&
\quad
\cdot
\bigg(
c \cos v+ d \sin v,
-c \sin v+d \cos v
\bigg)
 - n
\Bigg|
\\
&=&
\Bigg|
\rho
(a\cos u - b \sin u)
(
c \cos v+ d \sin v
)
\\
&&
+
\frac1{\rho}
(a\sin u + b \cos u)
(
-c \sin v+d \cos v
)
 - n
\Bigg|
\\
&\approx&
\Bigg|
\rho
(a\cos u - b \sin u)
(
c \cos v+ d \sin v
)
 - n
\Bigg|
\\
&<&
{N\over K}
.
\eeann

Let $\frak u$ be the angle between the vectors $(a,b)$ and $(\cos u, -\sin u)$. 
Similarly, let $\frak v$ be the angle between $(c,d)$ and $(\cos v, \sin v)$. Then the 
above becomes
$$
\bigg|
\rho
|v|
|w|
\cos\frak u
\cos\frak v
 - n
\bigg|
<
{N\over K}
,
$$
or
$$
\bigg|
\frac\rho T
\cos\frak u
\cos\frak v
 - 
 {n
 \over
|v|
|w|
 T}
\bigg|
<
{N\over |v|
|w|
 T K}
.
$$

Set
$$
\cA:= {n\over T|v||w|},\qquad\text{ and }\qquad
\cB:= {N\over T|v||w|}
.
$$
Both $\cos \frak u$ and $\cos\frak v$ 
can range in intervals independent of $K$, and hence so do $u$ and $v$. 
By an obvious approximation argument, divide 
 these intervals into sectors $u\in\Psi_{\ga}$ and  $v\in\Phi_{\gb}$. 
An application of Theorem \ref{thm:main} gives (using a smooth function to capture the lower bound)
\beann
&&\sum_{\g\in\G\atop  T(\cA-{\cB\over K})<\|\g\|< T(\cA+{\cB\over K})}
\bo\{u\in\Psi_{\ga}\}\bo\{v\in\Phi_{\gb}\} \\
&&
\qquad\qquad
\gg {1\over K}\left( \mu(\Psi_{\ga})\mu(\Phi_{\gb})c_{0}T^{2\gd} + \sum_{j}c_{j}T^{2s_{j}}\right)
+ O\left(T^{\frac 34 + 2\gd \frac 14}(\log T)^{1/4}\right)
.
\eeann

As $\Psi:=\cup\Psi_{\ga}$ and $\Phi:=\cup\Phi_{\gb}$ are intervals independent of $K$, we have $\mu(\Psi)\gg 1$ and same with $\mu(\Phi)$. This completes the proof.

\

\

\section{Proof of Theorem \ref{thm:UpBnd}}\label{sec:UpBnd}

Recall that we 
wish
to give an upper bound for the number of $\g\in\G$,
$\|\g\|< T$, with
$$
\big|
(c,d)\g- y
\big|
<
\frac NK
$$
and
$$
(c,d)\g\equiv y(\mod q).
$$
Here 
$y=(y_{1},y_{2})\in\Z^{2}$
 with $|y|< N$, $|(c,d)|<\frac NT$ and 
 $|y|<T|(c,d)|$.

Let $\G_{0}(q)$ be the subgroup of $\G$ (of level $q$) which stabilizes $(c,d)$ modulo $q$, that is $\g_{0}\in\G_{0}(q)$ iff $(c,d)\g_{0}\equiv (c,d)(q)$. Then we decompose $\g\in\G$ as $\g=\g_{0}\g_{1}$ with $\g_{0}\in\G_{0}(q)$ and $\g_{1}\in\G_{0}(q)\bk \G$. The count becomes
\bea
\nonumber
\hskip-.6in
&&
\sum_{\g_{1}\in\G_{0}(q)\bk \G}
\bo\{(c,d)\g_{1}\equiv y(q)\}
\sum_{\g\in\G_{0}(q)\cdot\g_{1}\atop\|\g\|<T}
\bo\big\{\big|
(c,d)\g- y
\big|
<
\frac NK
\big\}
\\
\label{modComp}
\hskip-.6in
&&
\ll
{1\over [\G:\G_{0}(q)]}
\sum_{\g\in \G\atop\|\g\|<T}
\bo\big\{\big|
(c,d)\g- y
\big|
<
\frac NK
\big\}
+
O\left(T^{\frac67 2\gd+\frac17 2\gT}\right)
,
\eea
where we used Theorem \ref{thm:unif} on the inner sum, and estimated
$$
\sum_{\g_{1}\in\G_{0}(q)\bk \G}
\bo\{(c,d)\g_{1}\equiv y(q)\}
\ll
1
.
$$
It remains to analyze
$$
\cN :=
\sum_{\g\in\G\atop \|\g\|<T}
\bo
\left\{|(c,d)\g-y|<\frac {N}K\right\}
.
$$

Writing $\g$ in $KA^{+}K$ coordinates, we have $\g=k_{u}a_{\rho}k_{v}$ with $1<\rho
\approx\|\g\|<T$.
The condition 
$
|(c,d)\cdot \g -y|<\frac N K
$
becomes 
\bea
\nonumber
\left(\frac N{K}\right)^{2}
&>&
|(c,d)\cdot \g -y|^{2}
\\
\nonumber
&=&
\bigg(
\rho
\cos v(c \cos u-d\sin u) - \frac1{\rho}\sin v( c\sin u + d \cos u)
-y_{1}
\bigg)^{2}
\\
\nonumber
&&
+
\bigg(
 \rho
\sin v(c \cos u-d\sin u)+ 
\frac1{\rho}\cos v( c\sin u + d \cos u)
-y_{2}
\bigg)
^{2}
\\
\nonumber
&\approx&
\bigg(
\rho
\cos v(c \cos u-d\sin u) 
-y_{1}
\bigg)^{2}
+
\bigg(
\rho
\sin v (c \cos u-d\sin u)
-y_{2}
\bigg)
^{2}
\\
\label{breakSqs}
&=&
\bigg(
\rho
|(c,d)|
\cos\frak u
-
|y|
\cos\frak v
\bigg)^{2}
+
|y|^{2}
(
1
-
\cos^{2}\frak v
)
,
\eea
after a calculation.
In the above,  we set
$$
c\cos u -d\sin u=(c,d)\cdot (\cos u,-\sin u) = |(c,d)|\cos\frak u
,
$$ 
where $\frak u$ is the angle between the vectors $(c,d)$ and $(\cos u,-\sin u)$,
and
$$
y_{1}\cos v+y_{2} \sin v=(y_{1},y_{2})\cdot (\cos v,\sin v) = |y|\cos \frak v
,
$$ 
where $\frak v$ is the angle between those two vectors.

By positivity, we break \eqref{breakSqs} into two pieces.
The piece
$$
\left(\frac N{K}\right)^{2} > |y|^{2}(1-\cos^{2}\frak v)
$$
requires
 $|\frak v|\ll {N\over |y| K}$. This forces $v$ to be contained in an interval, say $
\Phi$, of length $\ll {N\over |y| K}$.

The second piece simplifies to
$$
\bigg|
\frac\rho T
\cos\frak u
-
\cA
\bigg|
\ll
{
1
\over
K}
,
$$
where
$$
\cA
:=
{|y|\over T |(c,d)|
}
\cos\frak v
.
$$

As $\rho<T$, $\frak u$ ranges in a constant interval, say $\Psi$. We break into sectors as before and bound using Theorem \ref{thm:main}:
\beann
\cN
&\ll&\sum_{\g\in\G\atop \cA T(1-{c_{1}\over K})<\|\g\|<\cA T(1+{c_{1}\over K})}
\bo\{u\in\Psi\}\bo\{v\in\Phi\} \\
&\ll&
{1\over K} \mu(\Psi)\mu(\Phi) T^{2\gd} 
.
\eeann

 Since $|\Phi|\ll{1\over K}$, we have $\mu(\Phi)\ll \frac1{K^
{\gd}}$. Inserting the above into \eqref{modComp} and using $[\G:\G(q)]\gg q^{2}$ completes the proof.

\bibliographystyle{alpha}

\bibliography{../../AKbibliog}

\end{document}